\documentclass[11pt]{amsart}
\newtheorem{thm}{Theorem}
\newtheorem{lem}{Lemma}
\newtheorem{prop}{Proposition}
\newtheorem{cor}{Corollary}
\newtheorem{rem}{Remark}
\usepackage{amsfonts}
\usepackage{amsmath}
\usepackage{calligra}
\usepackage{amsmath}
\usepackage{amssymb}
\begin{document}
\newcommand{\tr}{\,^t\!}
\newcommand{\2}{\frac{\sqrt 2}{2}}
\newcommand{\e}{\mathcal E_{Sp_0(n,F) }(V_k)}

\newtheorem{example}{Example}

\title{On the construction of a finite Siegel space }
\author{ Jos\'{e} Pantoja$^{1,3}$, Jorge Soto Andrade$^{1}$, Jorge A.  Vargas$^{1,2}$}
\thanks {$^{1}$Partially supported by FONDECYT 1120578}
\thanks {$^{2}$Partially supported by  CONICET, SECYTUNC (Argentina) }
\thanks {$^{3}$Partially supported by Pontificia Universidad Cat\'{o}lica de Valpara\'{i}so}
\date{\today }
\address{Instituto de Matem\'{a}ticas, PUCV; }\address{Blanco Viel 596, Co. Bar\'{o}n, Valpara\'{i}so, Chile}
\email{jpantoja@ucv.cl}
\address{Departamento de Matem\'{a}ticas, Fac. de Ciencias, Universidad de Chile; } \address{Las Palmeras 3425 , \~{N}u\~{n}oa, Santiago, Chile}
\email{sotoandrade@u.uchile.cl}
\address{ FAMAF-CIEM; }\address{Ciudad Universitaria, 5000 C\'ordoba, Argentine}
\email{vargas@famaf.unc.edu.ar}



 \begin{abstract}
In this note we construct a finite analogue of classical Siegel's Space. Our approach is to look at it as a non commutative Poincare's half plane. The finite Siegel Space is described as the space of Lagrangians of a $2n$ dimensional space over a quadratic extension $E$ of a finite base field $F$. The orbits of the action of the symplectic group $Sp(n,F)$ on Lagrangians are described as homogeneous spaces. Also,  Siegel's Space is described as the set of anti-involutions of the symplectic group.

  \end{abstract}

\maketitle
\markboth{P-S-V}{Finite Siegel space}

\section{introduction}
Classical Siegel's half space is a clever generalization of Poincar\'{e}'s half plane. In \cite{siegel}, the starting idea is to replace the real base field $\mathbb R$ by the full matrix ring $M(n,\mathbb R)$. Then Siegel's half space consists of all symmetric complex $n \times n $ matrices whose imaginary part is positive definite.

We address here the case of a finite base field. Our approach to obtain the finite analogue of Siegel's half space is to extend the universal (double cover of) Poincar\'e's half plane construction given in \cite{jaja} to the case where the base field $F$ is replaced by a  ring  $A$ with involution denoted *, that we read "star".  A ring with involution is also called involutive ring, as in  \cite{psa, psa2}. Instead of the group
$G_F = SL(2,F)$ we have now its star-analogue $ G_A = SL_*(2,A)$ introduced in \cite{psa2}. A natural $G_A-$ space is the star-plane  $\mathcal P_A$ consisting of all points $x = (x_1, x_2)   \in P = A \times A$ whose coordinates  $x_1 $ and $x_2$  star-commute, i.e.  $x_1x_2^* = x_2x_1^*.$ Notice {\em en passant}  the analogy with Manin's $q-$plane, whose points have coordinates that anti-commute.

We introduce  the canonical star - anti-hermitian form $\omega$ on  $P$ given by
$$    \omega(x,y) =x_1y_2^* - x_2y_1^* $$ for all $x,y \in  P.$ We have then

 $$  \omega(y,x) =   - \omega(x,y)^* $$
 for all $x,y \in  P, $ and we see that the  star-plane  $\mathcal P_A$ consists of all isotropic vectors for  $\omega.$

 We also notice that if we write
 $$ x^*=  \left ( \begin{array}{c}
                        x_1^*  \\
                         x_2^*
                        \end{array} \right ) $$

 \noindent
 for    $ x = (x_1, x_2)  \in \mathcal P_A,$
 then  we  have
 $$\omega(x,y) = xwy^* $$
 where  $$ w =     \left(
     \begin{array}{cc}
       0 & 1 \\
          -1 & 0 \\
         \end{array}
         \right)          $$.

The star-plane $\mathcal P_A$ is stratified by the family of  $G_A-$ subplanes
$\mathcal P_A(J)$ of   $ \mathcal P_A$  given by the condition
$ Ax + Ay = J $ where $J$ is a left ideal in $A.$   In what follows we will be  mainly interested in the generic case  $J=A,$ and we will take $A = M_n(F)$ endowed with the transpose map.

As a motivation for the construction below, recall that  finite Poincar\'e half plane, more precisely the double cover of finite Poincar\'e half plane, may be realized as the set of lines through the origin in the usual plane  $E^2 =E\times E,  E $ a quadratic extension of the base finite  field  $F$,  whose slope does not lie in $F \cup \{ \infty\}. $    Lines through the origin are however just the Lagrangians for the symplectic bilinear form $determinant $ on  $E^2.$ and the constraint that the slope of a Lagrangian $L$ does not lie in $F \cup \{ \infty\}. $  amounts to say that the symplectic form $h$  given by Galois twisting of the determinant,  given by   $$ h(x,y) = x_1\bar y_2 - x_2 \bar y_1
   $$

\noindent
for    $ x =   \left ( \begin{array}{c} x_1 \\ x_2  \end{array} \right ), y = \left ( \begin{array}{c} y_1 \\ y_2  \end{array} \right )$ in $E,$
 is {\em non degenerate} when restricted to  $L.$

 Indeed, if the constraint on   $L$ is fulfilled, we may take a representative  vector of the form  $(z,1) \in L \;\; (z \in E),$ so that  $L = \{ (zx_2], x_2) | x_2 \in E \}$ and  then  $h$
 on $L$ is given by
$$ h( \left ( \begin{array}{c} zx_2 \\ x_2  \end{array} \right ),  \left ( \begin{array}{c} zy_2 \\ y_2  \end{array} \right )) =  (z - \bar z)x_2  \bar y_2,
   $$
so $h$ non degenerate means just  $ z \neq \bar z.$

Under the action of $SL(2,F)$ in the set $\mathcal L$ of all Lagrangians we have then the generic orbit consisting of all Lagrangians   on which $h$  is non degenerate and the residual orbit consisting of all Lagrangians   on which  $h$ is degenerate, i. e. null in this case, so that $ z = \bar z $, i.e. $ z \in F$.
More generally we will see below that the rank of  the restriction  $h_L$ of $h$ to $L$ characterizes the   $SL(2,F)-$ orbits in  $\mathcal L.$


\section{Preliminaries}

\subsection{General setup}
We assume now that  the involutive ring   $(A,*)$ is a quadratic Galois extension of a sub involutive ring $A_0$, i.e. that  the Galois group  $\Gamma = Gal(A, A_0)$ is of order 2 and that  $A_0 = Fix_A(\Gamma).$ We denote   $a \mapsto \bar a $ the nontrivial element $\tau$ of $\Gamma.$ Notice that $\tau$ extends naturally to the plane $A^2 = A \times A$  and to the star-plane  $\mathcal P_A,$   Our data is then  $(\mathcal P_A, \omega, \tau)$.

We introduce the canonical    star-$\tau$-antihermitian form $h$ on $P$ given by
$$     h(x,y) = \omega(x,\bar y) $$
for all $x,y \in  P. $  We have
$$   h(y,x) =  - \overline {h(x,y)}^*  $$
for all $x,y \in  P. $

\subsection{The full matrix ring case}
We specialize now to the case where the involutive ring $(A, *)$ is the full matrix ring $M_n(E)$ over a finite field $E$ endowed with the transpose mapping. We assume moreover that  $E$ is a quadratic extension of a subfield $F$ with Galois group $\{Id, \tau\}.$

We have  the big special linear group $G_E = SL_*(2, A) $ and the small special linear group
  $G_F = SL_*(2, A_0)$ that appears as the fixed point set of $\tau$ in $G_E$.
 The set of  all  lines through the   origin in the plane  $ \mathcal P_A$        is     denoted  by $\mathcal L_A$    It follows from classical  Witt's theorem    that $G_E$ acts transitively on $\mathcal L_A.$

 Indeed the non-commutative 1- dimensional subspaces $L \in \mathcal P_A$ may be readily identified with classical Lagrangians in the symplectic space $V =  E^{2n}, $ endowed with the canonical symplectic form $\omega',$ that in terms of the
   canonical basis  $e_1, \cdots, e_{2n}$ for $V$ is given by   $\omega (e_j, e_{n+j})=-\omega(e_{n+j}, e_j)=1, j=1, \dots, n$ and $\omega(e_k, e_s)=0 $ for $\vert k-s \vert \not= n.$

 Recall \cite{psa} that  Lagrangian subspaces  $L$  in $V$ may be described as
 $ L = L_{(a,b)} = \langle   aP + bQ  \rangle \;\;\;  (a, b \in A, aA+bA = A,  ab^* = ba^*) $
 \noindent
 where  the column vectors  $P$ and $Q$ are given by  $ P = (e_1, \cdots, e_{n})^*,    Q = (e_{n+1}, \cdots, e_{2n})^* $ and
 $\langle u \rangle $ stands for the vector subspace of $V$ spanned by the components   $ u_1,...,u_n$  of any $u \in  M = V^n.$

 Moreover
$L_{(a,b)} = L_{(a',b')} $   if and only if   $ a' = ca $  and  $ b' = cb$  for a suitable $ c  \in A.$
So classical Lagrangians correspond to non commutative lines through the origin in $\mathcal P_A.$

On the other hand  regarding the action of  $G_E$ we have  $ g(L_{(a,b)} )=L_{(a,b)g}g$ for 	 $ g \in G_E$.

  The set of classical Lagrangian subspaces for $\omega'$ in $V$ will be denoted  by   $\mathcal L_{V} $  or just  $\mathcal L.$

In what follows we will switch to  the classical setting for Lagrangians in $V$ for the case of  $ A = M_n(E)$.
\bigskip

 We denote the set of symmetric matrices with coefficients in $E^n$ by $Sym(E^n).$ The isotropy subgroup for the subspace $L_+$ spanned by $e_1, \dots, e_n$ is  the semidirect product of the subgroups
$$ K:=\{ \left( \begin{array}{cc}
  A & 0 \\
  0 & ^t\!A^{-1} \\
   \end{array}
   \right), A \in GL_n(E) \}$$
$$P^+=\{ \left(
     \begin{array}{cc}
       I & B \\
         0& I \\
         \end{array}
         \right),   B \in Sym(E^n)\}$$

On the other hand,  the isotropy subgroup for the subspace $L_-$ spanned by the vectors $e_{n+1}, \dots, e_{2n}$ is the semidirect product of $K$ times the subgroup
$$P^-=\{ \left(
     \begin{array}{cc}
       I & 0 \\
          B & I \\
         \end{array}
         \right),  B \in Sym(E^n)\}$$
         Let $\mathcal L : Sym(E^n) \rightarrow \mathcal L_{E,2n}$ be the Siegel map defined by the formula $$\mathcal L(Z)= \{\left ( \begin{array}{c}
                          Zx \\
                          x
                        \end{array} \right ) , x   \in E^n \}$$
We would like to point out that in  [8], a complete description of the Lagrangian subspaces of
$E^{2n}$, $E$   a finite field, is given, in the study of the groups $SL_{*}(2,A)$ (applied to $A=M_n(E)$ and $*$ the transposition of matrices).  The Siegel Lagrangian $\mathcal L(Z)$ is $L_{Z,I_{n}}$ in the notation of [8].\\

Following Siegel,  we write sometimes   $(A,B,C,D)$ for the $2n\times 2n$ matrix $$\left(
     \begin{array}{cc}
       A & B \\
          C & D \\
         \end{array}
         \right) \,\, A,B,C,D \in M_{n}(E )$$
  \begin{rem}        Whenever $F=\mathbb R,$  we have that $\mathcal L(Z)$ is equal  to the action on  the subspace $L_-$ of the exponential of the Lie algebra element $ (0,Z,0,0) \in \mathfrak{sp}(n, \mathbb C).$ \\  \noindent
  \end{rem}
 We define $\epsilon=\epsilon_F$  by

          $$\epsilon_F = \left\{ \begin{array}{ll}
      1 & \mbox{if $-1$ is an square in $F$} \\
        -1 & \mbox{otherwise} \\
       \end{array} \right.$$
   This is just the Lagrange symbol $(\frac{-1}{p})$ in the case of  a finite field $F$ of characteristic $ p.$

(We note that, in the real case, we always have $\epsilon_F=-1$)

\begin{prop}
We have the decomposition into $G_F-$ invariant subsets
$$  \mathcal L_V = \bigcup_{0\leq r \leq n} \mathcal H_r, $$
where $\mathcal H_r$ stands for the set of all $W \in \mathcal L_V$ such that the rank of $h_E$ restricted to $W\times W$ is $r.$
\end{prop}

 Next, we consider the hermitian form $$h_0 : E^{2n} \times E^{2n} \rightarrow E$$ defined so that the canonical basis is an orthogonal basis for  $h_0,$  $h_0(e_j,e_j)= -1$ for $1 \leq j\leq n$ and $h_0(e_j,e_j)=1 $ for $ n+1 \leq j \leq 2n.$

 We consider the group  $$Sp_0(n,F):= U(E^{2n},h_0)\cap Sp(n,E).$$  Later on, for a finite field $F$ we  construct a generalized Cayley transform, that is, we show there exists an element $C$ in $Sp(n,E)$ which conjugates  $Sp(n,F)$ into  $Sp_0(n,F).$ That is, $C^{-1} Sp_0(n,F)C =Sp(n,F).$
Thus, we verify that $Sp_0(n,F)$ is isomorphic to $Sp(n,F), ($a well known result for $F=\mathbb R$, see page 242 of \cite{jsatesis}).\\

  Among the objectives of this note  are, for a finite field $F$,  to determine the orbits   of both groups $Sp(n,F), Sp_0(n,F)$ in $\mathcal L_{E,2n}$ and the intersection of each orbit with the image of the Siegel map.  When $F=\mathbb R, E=\mathbb C$ this problem has been considered and solved by \cite{wolf2}, \cite{kaneyuki} and references therein.\\

\smallskip
 It is known that for a finite field $E$ and a hermitian form $(W,h)$ on a finite dimensional vector space $W$ over $E,$ there always exists an ordered  basis $w_1, \dots $ of $W$ and a nonnegative integer $r$ so that $h(w_k, w_s)=\delta_{ks} $ for $k,s\leq r$ and $h(w_k, w_s)=0 $ for $k >r$ or $s>r.$

 In this situation we define the type of  the form $(W,h)$  to be $r. $

\smallskip

Let $ \mathcal O_r$ the set of Lagrangian subspaces $W \in \mathcal L_{E,2n} $ so that the form $h_0$ restricted to $W$ is of type $r.$ Obviously $Sp_0(n,F)$ leaves invariant the subset $ \mathcal O_r$ and $\mathcal L_{E,2n}=\mathcal O_n \cup  \mathcal O_{n-1}\cup \cdots \cup  \mathcal O_0. $  \\
One of the main results of this work is:

\begin{thm} Assume $F$ is a finite field, then
\begin{itemize}
                                 \item The orbits of $Sp_0(n,F)$ in $\mathcal L_{E,2n} $ are exactly the sets $\mathcal O_j, j=0, \cdots, n.$
                                 \item The orbits of $Sp(n,F)$ in $\mathcal L_{E,2n} $ are exactly the sets $\mathcal H_j, j=0, \cdots, n.$
                                 \item Any orbit of either $Sp(n,F)$ or $Sp_0(n,F)$ intersects the image of the Siegel map.
                                 \item Except for $n=1$,  no orbit of $Sp_{0}(n,F)$ is contained in the image of the Siegel map.
                                 \item $\mathcal H_n$ is the unique orbit of $Sp(n,F)$ contained in the image of the Siegel map.
                                 \item $  C \mathcal H_j = C \mathcal O_j.$
                               \end{itemize}


\end{thm}

\section{Proofs}

In order to write down the proof of  theorem 1 we need to set up some notation and  recall some known facts.

   $\tr A $ denotes the transpose of the matrix $A.$ Vectors $v$ in $E^k $ are column vectors, so that we write   $\tr v$ for the  row vector corresponding to $v$

   In particular, we will use
  $$ E^{2n}\ni v = \left ( \begin{array}{c}
                          x \\
                          y
                        \end{array} \right ) , x , y \in E^n,    E^{2n}\ni w = \left ( \begin{array}{c}
                          r \\
                          s
                        \end{array} \right ) , r , s \in E^n,$$
\noindent
Let $I_n$ denote the $n\times n$ identity matrix and $0$ denotes the zero matrix. We set $$J:= \left( \begin{array}{cc}
  0 & I_n \\
  -I_n & 0 \\
   \end{array}
   \right).$$
Hence,  $\omega (v,w)= \tr x s - \tr y r= \tr v J w $.
Thus,  $$(A,B,C,D)=\left( \begin{array}{cc}
  A & B \\
  C & D \\
   \end{array}
   \right), A,B,C,D \in M_n(E)$$ belongs to $Sp(n,E)$ if and only if $$ \tr A C= ^t\!C A ,  \,\, \tr D B= \tr B D,  \,\, \tr A D- \tr C B=I_n.$$
Let $G_n(E^{2n})$ denote   the Grassmanian   of the $n-$dimensional subspaces of $E^{2n}.$ Hence, any of the  the groups $Sp(n,E), Sp(n,F), Sp_0(n,F)$ acts on $G_n(E^{2n})$ by $T W=T(W).$

A $n-$dimensional linear subspace $W$ of $(V,\omega)$  is a Lagrangian subspace  if and only if for every $v,w \in W, \tr v J w =0$ if and only if $ \tr x s-\tr y r=0$ for every $v,  w \in W.$ We fix $R,S \in E^{n\times n}$ and consider the subspace  $W=\{ \left ( \begin{smallmatrix} R x \\ Sx \end{smallmatrix} \right ), x \in E^n \}.$ Then, $W$ is Lagrangian if and only if $^t\!R S - \tr S R=0$ and the matrix $\left ( \begin{smallmatrix} R  \\ S  \end{smallmatrix} \right )$ has rank $n.$  Actually, any Lagrangian subspace may be written as in the previous example (see also [8]).  Particular   examples of Lagrangian subspaces are $L_+, L_-, \mathcal L(Z),$  $( Z \in Sym(E^n)).$ Needles to say,  the image of $\mathcal L$ is equal to the orbit $L_-$ under the subgroup $P^+,$  hence, Bruhat's decomposition yields
 that the  image of $\mathcal L$ is "open and dense" in $\mathcal L_{E, 2n}. $ Let $p : E^{2n} \rightarrow E^n$ denotes projection onto the second component. That is,  $p \left ( \begin{smallmatrix} x \\ y \end{smallmatrix} \right )=y .$ It easily follows
 that:

\noindent
A subspace $W \in G_n(E^{2n})$ belongs to the image  of $\mathcal L$ if and only if $W$ is Lagrangian and $p(W)$ is equal to $E^n.$ We  are ready for,
\begin{lem} Let $G$ be either $Sp(n,F)$ or $Sp_0(n,F)$ and fix $Z \in Sym(E^n)$. Then the  orbit  $G \mathcal L(Z)$ is contained in the image of $\mathcal L$ if and only if for every $(A,B,C,D) \in G$ the matrix $(CZ+D)$ is invertible.
\end{lem}
 {\it Proof:}   The subspace   $ (A,B,C,D) \mathcal L(Z)= \{ \left (\begin{smallmatrix} (AZ+B)x \\ (CZ+D)x \end{smallmatrix} \right ) , x \in E^n \}$ \\ is $n-$dimensional, Lagrangian  and  its image under $p$  is equal to the image of $CZ+D.$ Hence,  if $(CZ+D)$ is invertible, by a change of variable we have that $ (A,B,C,D) \mathcal L(Z) $ is equal to $\mathcal L(Z_1)$ for  $Z_1= (AZ+B)(CZ+D)^{-1}.$   Conversely, if the orbit $G \mathcal L(Z)$ is contained in the image of the Siegel map,   for each  $g=(A,B,C,D) \in G $ there exists $ Z_g \in Sym(E^{n}) $ so that $$\{ \left ( \begin{smallmatrix} (AZ+B)x \\ (CZ+D)x \end{smallmatrix}  \right ), x \in E^{n} \} = \{ \left ( \begin{smallmatrix} Z_g x \\ x \end{smallmatrix} \right ), x \in E^{n} \}.$$ Thus, the image of $CZ+D$ is equal to $E^n.$   \begin{flushright} $\Box $ \end{flushright}

\begin{cor} $(A,B,C,D) \mathcal L(Z) $ belongs to the image of $\mathcal L$ if and only if $(CZ+D)$ is an invertible matrix.
\end{cor}

\smallskip



\begin{example}
Orbits of $Sp_0(1,F)$ in the space of Lagrangians $\mathcal L_{E,2}.$ We assume $F$ is a finite field.  
 Let $N(e)=e \bar e$ be the norm of the extension $E/F.$ The hypothesis on $F$  implies  $N$ is a surjective map onto $F.$      After a computation, we obtain that $Sp_0(1,F)$ is the set of matrices
$$ \{\begin{pmatrix}
 \alpha & \beta \\
   \bar \beta &   \bar \alpha \\
   \end{pmatrix} : \alpha , \beta \in E, \alpha \bar \alpha -\beta \bar \beta=1 \}.$$

 In this case $\mathcal L_{E,2}=G_1(E^2),$ a typical element of $G_1(E^2)$ is  denoted by $E \left ( \begin{smallmatrix}
a \\
b\\
\end{smallmatrix} \right )$ with $a\not= 0$ or $b \not=0.$ Since $h_0 (\left ( \begin{smallmatrix}
z \\
1\\
\end{smallmatrix} \right ),\left ( \begin{smallmatrix}
w \\
1\\
\end{smallmatrix} \right ))= 1-z\bar w,$ it readily follows:
\begin{center}
$\mathcal O_1 =\{ E \left ( \begin{smallmatrix}
z \\
1\\
\end{smallmatrix} \right ), z \in E, N(z)\not= 1 \} \cup \{ L_+ \},$  \\ $\mathcal O_0  =\{ E \left ( \begin{smallmatrix}
z \\
1 \\
\end{smallmatrix} \right ), z \in E, N(z)= 1\}$
\end{center}
For $z$ so that $ N(z) \not= 1$  we have $ (1 -z\bar z)^{-1}=t  \bar t, t \in E.$
If we define the matrix $$A:= \begin{pmatrix}
 \bar t & zt \\
  \bar z \bar t &   t  \\
   \end{pmatrix} $$   then    $A \left ( \begin{smallmatrix}
0 \\
1\\
\end{smallmatrix} \right ) = \left ( \begin{smallmatrix}
 zt \\
 t \\
\end{smallmatrix} \right ) .$ Obviously $A \in Sp_0(1,F).$  We are left to transform $E\left ( \begin{smallmatrix}
0 \\
1\\
\end{smallmatrix} \right )$ into $E\left ( \begin{smallmatrix}
1 \\
0\\
\end{smallmatrix} \right ).$  For this, we fix $z \not=0$ such that $N(z^{-1})\not= 1.$ Then, by means of $A$ the line $E\left ( \begin{smallmatrix}
1 \\
0\\
\end{smallmatrix} \right )$
  is transformed into
the   line $E \left ( \begin{smallmatrix}
1 \\
\bar z  \\
\end{smallmatrix} \right ),$ which is equal to the line $E\left ( \begin{smallmatrix}
 \bar{z}^{-1} \\
 1 \\
\end{smallmatrix} \right ) .$  From the previous calculation the last line is transformed into the line $E \left ( \begin{smallmatrix}
0 \\
1\\
\end{smallmatrix} \right )  .$             Thus,          $Sp_0(1,F) $ acts transitively in $\mathcal O_1.$

   We now show $Sp_0(1,F)$ acts transitively in  $\mathcal O_0.$

   We fix $E \left ( \begin{smallmatrix}
a \\
1\\
\end{smallmatrix} \right )  $  so that $a \bar a=1.$ Let $E \left ( \begin{smallmatrix}
b \\
1\\
\end{smallmatrix} \right )  $   in $\mathcal O_0.$ Then $N(a)=N(b),$ owing to  theorem 90 of Hilbert we have $\frac{a}{b}=d \bar {d}^{-1}.$ Since, the characteristic of $F$ is different from two, the pair of vectors $ \left ( \begin{smallmatrix}
a \\
1\\
\end{smallmatrix} \right ) ,  \left ( \begin{smallmatrix}
1 \\
 -\bar a\\
\end{smallmatrix} \right )  ,$   as well as $ \left ( \begin{smallmatrix}
b \\
1\\
\end{smallmatrix} \right ) ,  \left ( \begin{smallmatrix}
1 \\
- \bar b\\
\end{smallmatrix} \right )  $  determine two  ordered  basis for $E^2.$ .
Let $T$ be the linear operator defined by
$T( \left ( \begin{smallmatrix}
a \\
1\\
\end{smallmatrix} \right ))=d \left ( \begin{smallmatrix}
b \\
1\\
\end{smallmatrix} \right )$ and $T( \left ( \begin{smallmatrix}
1 \\
- \bar a\\
\end{smallmatrix} \right ))=d^{-1}\left ( \begin{smallmatrix}
1 \\
- \bar b\\
\end{smallmatrix} \right ) $. One checks that $ h_0( T \left ( \begin{smallmatrix}
a \\
1\\
\end{smallmatrix} \right ) ,  T \left ( \begin{smallmatrix}
1 \\
- \bar a\\
\end{smallmatrix} \right )) = h_0( d    \left ( \begin{smallmatrix}
b \\
1\\
\end{smallmatrix} \right ), d^{-1} \left ( \begin{smallmatrix}
1 \\
- \bar b\\
\end{smallmatrix} \right ))$  and that $ \omega( T \left ( \begin{smallmatrix}
a \\
1\\
\end{smallmatrix} \right ) ,  T \left ( \begin{smallmatrix}
1 \\
- \bar a\\
\end{smallmatrix} \right )) = \omega( d    \left ( \begin{smallmatrix}
b \\
1\\
\end{smallmatrix} \right ), d^{-1} \left ( \begin{smallmatrix}
1 \\
-\bar b\\
\end{smallmatrix} \right ))$  to conclude that  $T$ lies in $U(E^2, h_0) \cap Sp(E^2,\omega)=Sp_0(1,F).$   Hence, $\mathcal O_0$ is  an orbit of $Sp_0(1,F).$
\end{example}

\bigskip

\begin{rem}
 The orbit $\mathcal O_0$ is contained in the image of the Siegel map, whereas the orbit $\mathcal O_1$ does contain a point in the complement to the image of the Siegel map. This observation shows that for a finite field $F$ and $n=1$ our conclusions are in concordance with  the results obtained by other   authors for the case of $F=\mathbb R.$ More precisely in the real case, $\mathcal O_1 $ splits  in the union of two orbits, one orbit is the set of lines where $h_0$ is positive definite and the other is the set of lines where $h_0$ is negative definite. In this case the orbit corresponding to the set of lines where $h_0$ is positive definite is contained in the image of the Siegel map, whereas the orbit corresponding to the set of lines where $h_0$ is negative definite is not contained in the image of the Siegel map. The orbit corresponding to the set
of lines where $h_0$ vanishes is contained in the image of the Siegel map.
\end{rem}

\bigskip

\begin{rem}
 The previous  computations together with  corollary 1 to lemma 1, let us conclude that $ \bar \beta z + \bar \alpha$ is nonzero for every element of $Sp_0(1,F)$ such that $- z \bar z +1 =0.$  Whereas, for each $z$ so that $- z \bar z +1 \not= 0,$ there exist an element of $Sp_0(1,F)$ so that $\bar \beta z + \bar \alpha =0,$ it is  the element that carries the line of direction $(z,1)$ onto the line of  infinite slope!
\end{rem}
\bigskip

 We have

 \begin{lem}  $Z$ be an element of  $Sym  (E^n). $ Then $ \mathcal L(Z) $ belongs to $\mathcal H_r$ if and only if the anti-hermitian form on $E^n$ defined by $ Z-\bar Z $   has rank $r.$
\begin{proof} For the non-degenerate anti-hermitian  form $h_E$ on $E^{2n},$ given by $h_E (v,w):=w(v,  \bar w)= \,^t\!x \bar s - \,^t\!y \bar r $ ($v,w \in E^{2n}),$ we have $Sp(n,F)=U(E^{2n}, h_E) \cap Sp(n,E).$ Hence, $\mathcal H_j$ is invariant under the action of $Sp(n,F).$  It follows that  $$ h_E( \left ( \begin{smallmatrix}
Zx \\
x\\
\end{smallmatrix} \right ), \left ( \begin{smallmatrix}
Zy \\
y\\
\end{smallmatrix} \right ))= \tr x (Z-\bar Z) y \;\;\;\; (x,y \in E^n),  $$ from which the result. \end{proof}
\end{lem}

\bigskip

\begin{example}
 We now compute the orbits of $Sp(1,F)$ in $\mathcal L_{E,2}$ for a finite field $F.$   For this we show   that each $\mathcal H_j$ is an orbit of $Sp(1,F).$ In fact, $$\mathcal H_1= \{ E  \left ( \begin{smallmatrix}
z \\
1\\
\end{smallmatrix} \right ) : z-\bar z \not= 0\} \hskip 0.2cm \text{and}\,\, \mathcal H_0= \{ E  \left ( \begin{smallmatrix}
z \\
1\\
\end{smallmatrix} \right ) : z \in F\} \cup \{E\left ( \begin{smallmatrix}
1 \\
0\\
\end{smallmatrix} \right )\}.$$
Since $ J \in Sp(n,F)$ we have that  $E\left ( \begin{smallmatrix}
1 \\
0\\
\end{smallmatrix} \right )$ is in the orbit of $E\left ( \begin{smallmatrix}
0 \\
1\\
\end{smallmatrix} \right ).$ Since the matrix $(1,t,0,1)\in Sp(1,F)$ and $(1,t,0,1) (0,1)=(t,1)$ we have that $Sp(1,F)$ acts transitively in $\mathcal H_0.$

We now show that $Sp(1,F)$ acts transitively in $\mathcal H_1.$ Let $ E  \left ( \begin{smallmatrix}
z \\
1\\
\end{smallmatrix} \right ), E  \left ( \begin{smallmatrix}
w \\
1\\
\end{smallmatrix} \right )$ so that $z-\bar z \not=0, w -\bar w \not=0,$ Since $F$ is a finite field, there exists $t_0 \in E$ so that $z-\bar z =t_0 \bar t_0 (w -\bar w).$  We define

$$ A:=\frac{1}{z-\bar z} \begin{pmatrix}
 t_0 w -\bar t_0 \bar w & z\bar t_0 \bar w - \bar z t_0 w \\
  t_0 -\bar t_0 &   z\bar t_0 -\bar z t_0\\
   \end{pmatrix}$$  The coefficients of $A$ belong to $F$ and  $$A\left ( \begin{smallmatrix}
z \\
1\\
\end{smallmatrix} \right ) =  \frac{z}{z-\bar z} \begin{pmatrix}
 t_0 w -\bar t_0 \bar w  \\
  t_0 -\bar t_0 \\
   \end{pmatrix} + \frac{1}{z-\bar z} \begin{pmatrix}
 z\bar t_0 \bar w - \bar z t_0 w \\
     z\bar t_0 -\bar z t_0\\
   \end{pmatrix} = t_0 \begin{pmatrix}
  w \\
     1\\
   \end{pmatrix}.$$  $$det A= \frac{(z-\bar z)(w-\bar w)t_0\bar t_0}{(z-\bar z)^2} =1.$$

   We note that $\mathcal H_1$ is contained in the image of the Siegel map, whereas $\mathcal H_0$ is not.
\end{example}

 \bigskip

\smallskip
We consider now

Let $g \in Sp_0(n,F),$ then  $$ g^{-1}=diag(- I_n, I_n)\,\, ^t\!\bar g \, diag(- I_n,I_n)$$  Therefore, the elements of
$Sp_0(n,F)$ are the matrices
 $$ \begin{pmatrix} A  & B \\ \bar B & \bar A \\
 \end{pmatrix} \, A, B \in M_{n}(E), \,\,  \tr \bar A  B = \tr B  \bar A,\,\, \tr A \bar{A}-\,\, \tr \bar B   B=I$$

 Since $Sp_0(n,F)$ is invariant under the map $ g \mapsto \tr g,$ we get the characterization of $Sp_0(n,F)$ obtained by \cite{siegel}, namely,
    \smallskip

    \noindent
    $(R,S,T,V) \in Sp_0(n,F)$ if and only if $$T= \bar S,\,\, V=\bar R, \,\, R\, \tr S= S\, \tr R, \,\,\,\, R\, \tr \bar R -S\,\, \tr \bar S =I_n . \eqno(S)$$
    as is readily seen.
    \medskip

  \smallskip

A simple computation shows:   $$Sp_0(n,F) \cap K P^+ = Sp_0(n,F) \cap K P^- = diag(A, \bar A), A \in
U(n,E).$$
\medskip

\smallskip

Now assuming that  $F$ is a finite field, we prove that any set $\mathcal O_r$ intersects nontrivially the image of
the Siegel map, and for $r>0$,  that $\mathcal O_r$ contains a point of the complement of the image of
the Siegel map.
\smallskip
\medskip

We observe that the form $h_0$ restricted to  $\mathcal L(diag(d_1, \dots, d_n)) $ is the diagonal form $$ (1- d_1 \bar d_1)x_1 \bar y_1 + \cdots + (1- d_n \bar d_n)x_n \bar y_n.$$ Thus,  $F$ being  a finite field,  allow us to find  $d$ so that  $d \bar d =1,$ from which we obtain  that  $\mathcal L(diag(0,\dots,0, d,\dots d))$ ($r$ zeros) belongs to $\mathcal O_r. $
\smallskip
\medskip

We fix now $0<r \leq n$ and $d \in E $ such that $d\bar d=1.$ Let $W_r$ denote the subspace spanned by the vectors $e_1, \dots, e_r, d e_{r+1}+ e_{n+r+1}, \dots,  d e_n +e_{2n}.$ Then $W_r$ is $n-$dimensional and isotropic for $\omega.$ The matrix of the form $h_0$ restricted to $W_r$, on the above basis,  is $diag(-1, \dots, -1, 0, \dots, 0),$ (here $-1$ occurs $r$-times). Hence, $W_r $ belongs to $\mathcal O_r.$ Moreover, the dimension of  $p(W_r)$ ($p$ as defined before lemma 1) is  $n-r <n.$ Therefore $W_r$ does not belong to the image of the Siegel map.

\bigskip

\noindent
\begin{rem}  For any permutation matrix $T$ we have that the matrix $\begin{pmatrix}  \tr T^{-1}  & 0 \\ 0 & T\\
 \end{pmatrix}   $ belongs to $Sp(n,F)_0 \cap Sp(n,F).$
\end{rem}
For the time being we assume  $-1$ is not a square in $F.$
\medskip
\smallskip

 We now show that  for odd $n>1$ ,   $\mathcal O_0$  contains points in the image of the Siegel map, and contain points in the complement of the image of the Siegel map.

  To begin with, we consider  $n=3.$

We fix $d, c \in E$ so that $ 0= 1+c\bar c  + d \bar d$ and $c\bar d \in F.$

We set $$ A:= \begin{pmatrix}  1 &  0 & -c \\ 0  & 1 & -d \\   \bar c  & \bar d & 1 \\
 \end{pmatrix} \,\,\, B:= \begin{pmatrix}  1 & 0 & 0 \\ 0 & 1 & 0\\ c & d & 0\\
 \end{pmatrix}     $$

Then, $$\tr A B = \begin{pmatrix}  1+c\bar c &  \bar c d & 0 \\ c\bar d & 1+d\bar d & 0 \\   0  & 0 & 0 \\
 \end{pmatrix}, \,\, \tr B A = \begin{pmatrix}  1+c\bar c &   c \bar d & 0 \\ \bar c d & 1+d\bar d & 0 \\   0  & 0 & 0 \\
 \end{pmatrix} $$ Given that $\bar c d \in F,$  both matrices are equal. Thus, the subspace $L:=\{(Ax,Bx), x \in E^n \}$ is Lagrangian.

 Since

 $$\tr A \bar A = \begin{pmatrix}  1+c\bar c &  \bar c d & 0 \\ c\bar d & 1+d\bar d & 0 \\   0  & 0 & c\bar c +d\bar d+1 \\
 \end{pmatrix}, \,\, \tr B \bar B = \begin{pmatrix}  1+c\bar c &   c \bar d & -c \\ \bar c d & 1+d\bar d & 0 \\   0  & 0 & 0 \\
 \end{pmatrix} $$  both matrices are equal, and therefore $h_0$ restricted to $L$ is the zero form. On the other hand,  $det A=1+d\bar d +c \bar c= det B=0.$ This shows that  $L$ is an element of $\mathcal O_0$ which is not in the image of the Siegel map.
 \smallskip

Now, in order to produce an element of $\mathcal O_0$ in the complement of the image of the Siegel map for odd $n$ with  $n >3$ , we write $n=3+n-3.$ Then the subspace $L\oplus E( e_4 +e_{n+4}) \oplus \dots \oplus E( e_n + e_{2n}) $ satisfies our requirement.
\smallskip

Finally, the subspace $\mathcal L(I_n),$ is an element of $\mathcal O_0$ which is in the image of the Siegel map.

\bigskip


\smallskip

 For $n$ even, $\mathcal O_0$ contains points in the complement of image of the Siegel map.

Let us take $c,b \in E$ such that $ b \bar b=-1.$ We set  $$ A:= \begin{pmatrix}  -bc  & -b \\ c & 1\\
 \end{pmatrix} \hskip 1.0cm  B:=\begin{pmatrix}  1  & 0 \\ b & 0\\
 \end{pmatrix}    $$ Then $\tr A B= \tr B A=(0,0,0,0)$

 Thus, $W :=\{(Ax,Bx), x \in E^2 \}$ is a Lagrangian subspace.

 Given that $ \tr A \bar A =\tr B \bar B =(0,0,0,0),$ we see that, $h_0$ restricted to $W$ is the null form, that is, $W \in \mathcal O_0.$

 Further, neither $A$ nor $B$ is invertible, and so $W$ is not in the image of the Siegel map.

  For $n=2k,$ it readily follows that the subspace $W \oplus \dots \oplus W_2$ ( $k-$times) belongs to $\mathcal O_0$ and it does not belong to the image of the Siegel map.

\bigskip


\smallskip

We  compute now an  example of  points in   $\mathcal O_n$ which are  outside the image of the Siegel map, and also compute an element in $Sp_0(n,F)$ which carries these points into the image of the Siegel map.

 For this, we notice that $(\alpha I_n, \beta I_n,  \bar \beta I_n, \bar \alpha I_n)$ belong to $ Sp_0(n, F)$ if and only if $\alpha \bar \alpha - \beta \bar \beta=1.$

  We fix  an integer $k$  so that $1 <k<n$. The subspace $Z_k$  spanned by $e_1, \dots, e_k, e_{n+k+1}, \dots, e_n$ is Lagrangian, and $h_0$ is non degenerate on it. Obviously, $Z_k$ does not belong to the image of the Siegel map.
  We may  choose nonzero $\alpha, \beta $ so that $\alpha \bar \alpha - \beta \bar \beta=1.$ Then  $(\alpha I_n, \beta I_n,  \bar \beta, \bar \alpha)$ takes $Z_k$ into a subspace which belongs to the image of the Siegel map.

\bigskip
We will use bellow the following involution: for a matrix $A$, $A^\star=^t\! \bar A$
\begin{lem}  $Sp_0(n,F)$ acts transitively on $\mathcal O_n.$
\end{lem}
\begin{proof}
We have that $L_-=\mathcal L(0)$ is an element of $\mathcal O_n.$  First, we will prove that given $\mathcal L(Z) \in \mathcal O_n $,  there is an element of $Sp_0(n,F)$ which
carries $\mathcal L(Z)$ onto $L_-.$

The matrix of the form $h_0$ restricted to $\mathcal L(Z)$  is $I_n
- Z \bar Z$. Choosing an adequate basis, there exists an
invertible  matrix $A$ so that  $A (I_n - Z \bar Z)\, ^t\! \bar A=  I_n.$
Let set $B:=-AZ.$ Then, since $$A\,\, ^t\!(-AZ)=-AZ\, ^t\!A, \,\text{and}\,  A\,\, ^t\!\bar A -(-AZ)(-\, ^t\!(\bar{AZ})=A (I_n -Z \bar Z)\, ^t\! \bar A=  I_n$$ the
matrix  $(A,B,  \bar B, \bar A)$  belongs to $Sp_0(n,F)$ (it satisfies (S)).

On the other hand,  $$(A,B,C,D)\mathcal L(Z)=\{ \left ( \begin{smallmatrix}
(AZ+(-AZ))x \\
( \bar B Z+\bar A)x\\
\end{smallmatrix} \right ), x\in E^n \}=\{ \left ( \begin{smallmatrix}
0 \\
 \bar A (I_{n}- \bar Z Z) x\\
\end{smallmatrix} \right ), x\in E^n\},$$ By above, the matrix $  \bar A (I_{n}- \bar Z Z)$ is invertible, so that $\mathcal L(Z)$ belongs to the orbit of $L_-.$

\smallskip

\noindent
Next, we will show that if $W=\{ \left ( \begin{smallmatrix}Rx \\ Sx \end{smallmatrix} \right): x \in E^n \} \in \mathcal O_n,$ then there exists an element $g$ in $Sp_0(n,F)$ so that
$g W \in Image(\mathcal L).$

In fact, we will show there exists $g \in Sp_0(n,F)$  so that $g W= \{ (Cx,Dx):\, x \in E^n \}$ with $C$ invertible, and then by means of a matrix $(0,dI_n, \bar d I_n, 0)$ we will transform $gW$ into an element of the image of the Siegel map.

Since $ W $ is in $ \mathcal O_n, $ there exists an invertible matrix $A$ such that $$ A (- R^\star R + S^\star S) A^\star = I_n.$$ Let us consider  $g=(- A R^\star, AS^\star,  \bar{AS^\star}, \bar{-A R^\star}). $ Then $$g W= \{ ((A^{\star})^{-1}x, ( \bar{AS^\star} R- \bar{AR^\star}S) x) x \in E^n \}.$$ Since   $$ - A R^\star (- A R^\star)^\star - AS^\star (AS^\star)^\star = A (-R^\star R + S^\star S) A^\star = I_n$$
   Also $ ^t\!R S= \, ^t\!S R,$ (because $W$ is a Lagrangian subspace), hence  we have $- A R^\star \,\, ^t\!(AS^\star) = -A R^\star \bar S \, ^t\!A= - A \, \bar{^t\!S} \bar{R} \, ^t\!A= AS^\star \, \, ^t\!(AR^\star),$ and so the matrix $g$ belongs to $Sp_0(n,F)$This concludes the proof that $\mathcal O_n$ is the orbit of $L_-$ under the group $Sp_0(n,F).$
\end{proof}
\bigskip

\begin{prop} There exists  element $C $ in $Sp(n,E)$ so that $C^{-1}$ conjugates $Sp_0(n,F)$ onto $Sp(n,F).$
\end{prop}


\begin{proof} When $-1$ is not an square in $F$ the proof follows quite close to the real case. We fix $i \in E $ a square root for $-1.$  We consider the $2n\times 2n$ matrix  $$ C_n :=\frac{1}{\sqrt{-2}} \begin{pmatrix} i I_n  & I_n \\ I_n & i I_n \\
 \end{pmatrix}$$ It readily follows that the matrix $C_n \in Sp(n,E).$ Let $\tau_F$ denote $-1$ if $-2$ is not a square in $F$ and $1$ if $-2$ is a square in $F.$  We now verify the equality $$\tau_F i h_E (v,w)= h_0 (C_n v, C_n w) \,\, \text{for \,every} \, v,w \in E^{2n}.$$ For this, we note that $$ ^t\!C_n diag(-I_n, I_n) \bar C_n = \tau_F i J, \hskip 1.0cm \bar C_n J^{-1} \,^t\!C_n = i\tau_F diag(-I_n, I_n)$$ Indeed,    $$^t\!C_n diag(-I_n, I_n) \bar C_n= \frac{1}{-2 \tau_F } \left ( \begin{smallmatrix}
-i & 1 \\
-1 & i\\
\end{smallmatrix} \right ) \left ( \begin{smallmatrix}
-i & 1 \\
1 &-i\\
\end{smallmatrix} \right ) = \frac{-2i}{-2 \tau_F } J. $$
$$ \bar C_n J^{-1} \,^t\!C_n = \frac{1}{-2\tau_F}  \left( \begin{smallmatrix}
-i & 1 \\
1 &-i\\
\end{smallmatrix} \right )  \left ( \begin{smallmatrix}
-1 & -i \\
i & 1\\
\end{smallmatrix} \right ) =i\tau_F diag(-I_n, I_n).$$ Hence, for $g \in Sp(n,F)$ we have  $$ \,^t\!(C_n g C_n^{-1}) diag(-I_n, I_n) \overline{CgC^{-1}} = \tau i \, \,^t\!C^{-1} \, \,^t\!g J g \bar C^{-1}$$ $$  =\tau i (\tau i\,\, diag(-I_n, I_n))^{-1} = diag(-I_n, I_n).$$ Thus, $CgC^{-1} \in Sp_0(n,F).$  Owing to the equalities of above we deduce, $h_0(v,w)=h_0(C g C^{-1} v, Cgc^{-1} w)= \tau_F i h_E(v,w).$ Tracing back the computation, we arrive to $C^{-1}gC \in Sp(n,F)$ for $g \in Sp_0(n,F).$ Hence, we have proved the proposition when $-1 \notin F .$  In case  $-1 \in F$ we follow the proof in \cite{jsatesis}. We choose $ v \in E $ so that $N(v)=-1, b \in E^\times   : b+\bar b=0.$ We define  $$ C_n := \frac{1}{\sqrt{b(v^2-1)}}  \begin{pmatrix} v I_n  & b I_n \\ I_n & vb I_n \\
 \end{pmatrix}$$ Then, $ C_n \in Sp(n,E) $ and $ (b(v^2-1)) \tr \bar C_n C_n =(v+\bar v) b J.$  A similar computation gives $h_0( Cv, Cw)= -(v+\bar v) b\, h_E(v,w).$
\end{proof}

 \begin{cor}  The group $Sp(n,F)$ acts transitively on $\mathcal H_n.$

\end{cor}
\begin{proof}
 Since the groups $Sp(n,F)$ and $Sp_0(n,F)$ are conjugated by the Cayley transform and the Cayley transform is a conformal map for the pair of bilinear forms $h_0, h_E$ the corollary follows 
\end{proof}

\smallskip
\noindent

 \begin{rem}  If $-1 $ is not a square in $F.$
$$C_n^{-1} =-\tau_F \bar C_n.$$
\end{rem}

\smallskip

For a subset  $W$ of $E^{2n},$ we define $\overline W=\{ \bar w , w \in W \}.$ For the linear subspace $W,$  we denote by $r_W$ the rank  of the form $h_E$ restricted to $W.$
\smallskip
\begin{lem}  For a Lagrangian subspace $W$ of $E^{2n}$ we have:  $$ dim (W + \overline W)=n+r_W$$   $$ dim(W \cap \overline W)= n-r_W$$
Furthermore, $ W \cap \overline W= (W + \overline W)^{\perp_{\omega}} =
W^{\perp_{h_E}}.$

\end{lem}

\begin{proof} We use the identities $$Z^{\perp_\omega} \cap U^{\perp_\omega}= (Z+U)^{\perp_\omega}, (Z \cap U)^{\perp_\omega}= Z^{\perp_\omega} +U^{\perp_\omega} .$$ Since $W,\overline W$ are Lagrangian subspaces we have $$ W \cap \overline W =W^{\perp_\omega} \cap \overline W^{\perp_\omega} = (W +\overline W)^{\perp_\omega}.$$  Fix $ y=\bar z \in  W \cap \overline W, z \in W, \text{and} \, x \in W,$ hence $ h_E(x,y)=\omega ( x, \bar y) = \omega (x, z) =0.$ Hence, $y \in W^{\perp_{h_E}}.$ Next, for  $ y\in W^{\perp_{h_E}}, $  we have $ \omega (\bar x, y)=0 $ for every $x \in W.$ The hypothesis $W$ is Lagrangian forces $ \bar y \in W,$ hence $ y=\bar{ \bar y} \in W \cap \overline W.$
\end{proof}

\begin{prop} For a finite field $F$ and $k=0,\dots,n,$ the group  $Sp(n,F)$ acts transitively on $\mathcal H_k.$

\end{prop}

\begin{proof}  We make the following induction hypothesis: for every $m<n$ and for every $k\leq m$ the group $Sp(m,F)
$ acts transitively on the $\mathcal H_k$ determinate by the corresponding form $h_E$ on  $(E^{2m}, \omega).$

Since, we have already  shown  that $Sp(1,F)$ acts transitively on $\mathcal H_k, k=0,1,$  the first step of the
induction process follows.

We recall also that for $n$ and $k=n$ we have shown that $Sp(n,F)$ acts transitively on
$\mathcal H_n.$
We are left to consider $r<n.$

\bigskip

We fix $W, Y \in \mathcal H_r$ with $r=r_W <n,$ we must find $g \in Sp(n,F)$ so that $g W=Y.$

Since, each of the subspaces $W \cap \overline W, W +\cap \overline W$ are
 invariant under the Galois automorphism, it follows that the subspaces are the complexification of, respectively,
$F^{2n}\cap W \cap \overline W, F^{2n}\cap (W +\cap \overline W).$ We notice that the
 quotient space $(W +\cap \overline W)/ (W \cap \overline W)$ is of dimension $n+r -(n-r)=2r <2n  . $
  Now, by  above we have that the push forward  to $(W + \overline W)/ (W \cap \overline W)$ of
  the form $\omega $ is a non degenerate form, and the same holds for $h_E.$

  Thus, the inductive hypothesis
  gives a linear transform $$ T : F^{2n}\cap (W + \overline W)/ ( F^{2n}\cap W \cap \overline W)
   \rightarrow F^{2n}\cap (Y + \overline Y)/ ( F^{2n}\cap Y \cap \overline Y)$$ such that $T^\star \omega
    =\omega$, and the complex  extension transforms $W/(W\cap \overline W)$ onto $Y/(Y\cap \overline Y).$
     We lift $T$ to a linear transform $$ T : F^{2n}\cap (W + \overline W) \rightarrow F^{2n}\cap
     (Y +\overline Y)$$ so that $T^\star \omega =\omega$ and the complex extension transforms $W$ onto $Y.$
      Now we apply the theorem of Witt to $T$ to get an element $g$ of $Sp(n,F) $ which carries $W$ into $Y.$
      This completes the induction process and we have the result
      \end{proof}

\begin{cor} $Sp_0(n,F)$ acts transitively in $\mathcal O_k, k=1,\dots,n$
\end{cor}

 \begin{lem}  $\mathcal H_n$ is   contained in the image of the Siegel map.
  \end{lem}
  \begin{proof} Let $ W \in \mathcal H_n.$ We may choose representatives $R$ and $S$ for $W$ and  write then $W=\{ \left ( \begin{smallmatrix} R x \\ Sx \end{smallmatrix} \right ), x \in E^n \} $ with   $^t\!R S - \tr S R=0,$  $\left ( \begin{smallmatrix} R  \\ S  \end{smallmatrix} \right )$ of rank $n $ .  Since $ W \in \mathcal H_n$ the matrix $\tr R \bar S - \tr S R $  is invertible.

The  matrix $S$ has rank $r$, with $0 \leq r  \leq n. $ We will show that $r=n.$

 We   choose two $n \times n$ permutation matrices  matrices $P,Q$ in $GL_n(F)$ such that $$ PSQ =\begin{pmatrix}  B_1  & B_2 \\ B_{3}& B_{4}\\
 \end{pmatrix}   $$ where $B_1$ is an invertible $r\times r$ matrix.  We may write $W= \{ \left ( \begin{smallmatrix} R Q x \\ SQ x \end{smallmatrix} \right ), x \in E^n \} $ and $$
 \begin{pmatrix}  (^tP)^{-1}  & 0 \\ 0 &  P\\
 \end{pmatrix} W =  \{ \left ( \begin{smallmatrix} (^tP)^{-1}R Q x \\ (B_1,B_2,B_3,B_4) x \end{smallmatrix} \right ), x \in E^n \} $$  Since $(B_1,B_2,B_3,B_4)$ has rank $r$ performing column operations, we may assume $B_2$ and $B_4$ are the zero matrices. This amounts to a new change of representatives for $W. $ Thus, $ W=  \{ \left ( \begin{smallmatrix} Ax \\ (B_1,0,B_3,0) x \end{smallmatrix} \right ), x \in E^n \} $

 Write $A:=  (A_1,A_2,A_3,A_4)$ with $, A_1 \in M_{r}(F), A_4 \in M_{n-r}(F). $  The hypothesis  $W$ is a Lagrangian, implies $ \tr A (B_1,0,B_3,0) = \tr (B_1,0,B_3,0)  A $ from which $ \tr A_1 +\tr A_3 B_3  = A_1 + \tr B_3 A_3  , A_2 + \tr B_3 A_4 =0.$   The hypothesis  the rank of $(A, (B_1,0,B_3,0))  $ is $n$ implies $A_4$ is invertible.  Hence, replacing $x$ by $ (diag(I_r), 0,0, A_r^{-1})x $ gives   $W=\{ \left ( \begin{smallmatrix} C x \\ Dx \end{smallmatrix} \right ), x \in E^n \} ,$ with $C=(A_1,0,A_3, I_{n-r})$ and  $D= (diag(I_r), 0,B_3, 0).$ The matrix of $h_E$ in this new coordinates is  $$\tr C \bar D - \tr D \bar C= \begin{pmatrix}  \bullet  & 0 \\ \bullet & 0\\
 \end{pmatrix}   .$$ The hypothesis $h_E$ restricted  $W$ has rank $n$ implies then $ n-r=0.$
 \end{proof}


\begin{cor} For any element  $Z \in Sym(E^{n}) $ such that $Z-\bar Z$ is invertible and for any $(A,B,C,D) \in Sp(n,F),$ the matrix $CZ+D$ is invertible.

\end{cor}

 We have completed the proof of theorem 1.

\medskip

Furthermore, we have the following facts:
\smallskip

\begin{rem}
For   a symmetric matrix $Z$ such that $Z -\bar Z$ is not invertible, there exists $(A,B,C,D),(M,N,R,S)  \in Sp(n,F)$ such that $CZ+D$ is invertible and $RZ+S$ is not invertible.

This follows  from  Corollary 1 to lemma 1  and theorem 1.
\end{rem}

\begin{rem}
For $n>1$  and   any symmetric matrix $Z$   there exists $(A,B,C,D), \\
(M,N,R,S)  \in Sp(n,F)_0$ so that $CZ+D$ is invertible and $RZ+S$ is not invertible.

This follows from lemma 1 and theorem 1
\end{rem}

 \section{Isotropy subgroups} The purpose of this section is to explicitly compute the structure of $\mathcal O_k, \mathcal H_k,  k=0, \dots, n$ as homogeneous spaces.  For the real case, this has been accomplished by \cite{takeuchi} \cite{wolf2} and references therein.

An element of $\mathcal O_{n-k} $ is constructed as follows: we define $V_k $ to be the subspace spanned by the vectors $ e_1 + e_{n+1}, \dots,  e_k +  e_{n+k}, e_{k+1}, \dots, e_{n}.$ Then, $V_0 =L_+$.  A simple computation shows that  the form $h_E$ restricted to $V_k \times V_k$ is the null form, whereas the type of the form $h_0$ restricted to $V_k \times V_k$ is $n-k$,. Obviously  $V_k$ is a lagrangian subspace.  Henceforth, for $x \in Sp(n,E)$,   $Ad(x)$ denotes the inner automorphism defined by $x.$ Let $t_k$ be the partial Cayley transform $$t_k := \begin{pmatrix}  D_1  & D_2 \\ D_{3}& D_{4}\\
 \end{pmatrix}   $$ where, $D_1= D_4= diag(\frac{\sqrt 2}{2} I_k, I_{n-k}), D_2 =  diag(-\2 I_k,0), D_3 =-D_2.$  Then, $t_k$ is an element of $Sp(n,E)$ and  $t_k L_+= t_k V_0 =V_k.$  A computation gives $$ t_k^{-1}= \begin{pmatrix}  L_1  & L_2 \\ L_{3}& L_{4}\\
 \end{pmatrix}   $$ where, $L_1= L_4= diag(\frac{\sqrt 2}{2} I_k, I_{n-k}), L_2 =  diag(\2 I_k,0), L_3 =-L_2.$

 Let $\mathcal E_{Sp_0(n,F)}(V_k)$ denote the set stabilizer of $V_k$ in $Sp_0(n,F).$
  The equality $ \mathcal E_{Sp(n,E) }(V_0) =KP^+$ implies $$\mathcal E_{Sp_0(n,F) }(V_k) = Ad(t_k) \mathcal E_{Sp(n,E) }(V_0)  \cap  Sp_0(n,F) = Ad(t_k) (KP^+)    \cap Sp_0(n,F) .$$

 The stabilizer of $V_0$ in $Sp(n,F)$ is $K P_+ \cap Sp_0(n,F) =K \cap Sp_0(n,F)= \{ diag(T, \bar T) : T \in U(n,E) \}$. Thus,  the stabilizer of $V_0$ in $Sp_0(n,F)$ is isomorphic to $U(n,E).$

     The main result of this section is
 \begin{thm} The stabilizer  group $\e$ is isomorphic to the semidirect product of the group $O(k,F) \times U(n-k, E)$ times the unipotent subgroup $ Ad(t_k)(P^+) \cap Sp_0(n,F).$

 \end{thm}

 The proof of the result requires some computations, which we carry out.

 First, we verify that the subgroup of $Sp_0(n,F)$,  $ diag(S,T, S, \bar T), S $ in $O(k,F), T$ in $ U(n-k,E)$ is contained in $\e.$ For this, we write for  $ v \in V_k, v= \begin{pmatrix} x \\ y \\ x \\ 0 \\ \end{pmatrix} $ with $x \in E^{k}, y \in E^{n-k}.$ Hence, $$ diag(S,T, S, \bar T) v= \begin{pmatrix} S x \\ Ty \\ Sx  \\ 0 \\ \end{pmatrix} \in V_k .$$ Is clear that the unipotent subgroup is contained in $\e.$

 For a matrix $T \in E^{n\times n}$ we write $$T =\begin{pmatrix} T_1 & T_2 \\ T_3 & T_4 \\ \end{pmatrix}, T_1 \in E^{k \times k}, T_2 \in E^{k \times n-k}, T_3  \in E^{n-k \times k}, T_4 \in E^{n-k \times n-k}$$ And for $(A,B,0,D) \in KP^+ $ we have \begin{multline*} Ad(t_k)(A,B,0,D) \\ = \begin{pmatrix} \frac12 (A_1 -B_1 +D_1) & \2 A_2 & \frac12 (A_1 +B_1 -D_1) & \2 (B_2 -D_2) \\ \2(A_3 -B_3) & A_4 & \2 (A_3 +B_3) & B_4 \\ \frac12 (A_1 -B_1 -D_1) & \2 A_2 & \frac12 (A_1 +B_1 +D_1) & \2 (B_2 +D_2) \\ -\2 D_3 & 0 & \2 D_3 & D_4 \\ \end{pmatrix} .\end{multline*}    Next, we show that $Ad(t_k) K \cap Sp_0(n,F)$ is equal to the subgroup \\ $\{  diag( S, T, S, \tr T^{-1}): S \in O(k,F), T  \in U(n,E)\}.$ In fact,  the computation for $Ad(t_k)X$ gives for $S \in O(k,F), T  \in U(n,E)$ that \\ $ Ad(t_k)(diag(S,T, S, \tr T^{-1}))= diag( S, T, S, \tr T^{-1}).$

 Now for  $ diag(A,D) = diag(A, \tr A^{-1})  \in K, $  such that $Ad(t_k)(diag(A,D)) \in Sp_0(n,F),$   (1.2) and the formula for $Ad(t_k)X$  imply  the equalities $$ \overline{ (A_1 + D_1)} = A_1 +D_1, \hskip 0.5cm \bar A_2 =D_2 \hskip 0.5 cm \bar A_3 = D_3, \hskip 0.5cm \bar A_4 = D_4 $$ and $$ \overline{A_1 - D_1} =A_1 -D_1, \hskip 0.5cm A_2 =-\bar D_2, \hskip 0.5cm A_3 =-\bar D_3 $$ So $$D_2=A_2=0,\hskip 0.5cm D_3 =A_3 =0, \hskip 0.5cm, \bar A_1= A_1, \bar D_1=D_1.$$ Hence, $A_1 \in O(n,F).$ Finally, the equality $D=\tr A^{-1}$ yields, $A_1= D_1,$  which shows the claim.

 Now $Ad(t_k)P^+ \cap Sp_0(n,F)=\{ Ad(t_k)(I_n, B,0,I_n) : \tr B =B \, \text{and} \,  \bar B_1=  -B_1, B_3=\tr B_2 =0 , B_4=0\}.$ In fact, the formula for  $Ad(t_k)X$ leads us

  to $$ Ad(t_k)(I,B,0,I) = \begin{pmatrix} \frac12 (2I -B_1) & 0 & \frac12 B_1 & \2 B_2  \\ -\2 B_3 & I & \2 B_3 & B_4 \\ -\frac12  B_1  & 0 & \frac12 (2I+B_1) & \2 B_2  \\ 0 & 0 & 0 & I \\ \end{pmatrix} $$

From (1.2) we get $\bar B_1 = -B_1, \hskip 0.3cm B_2=0, \hskip 0.1cm  \ B_4=0 ,$ and the equality follows.
\smallskip

(E) We will show at this point the equality \\
$\e =(Ad(t_k)K \cap Sp_0(n,F)) (Ad(t_k)P^+ \cap Sp_0(n,F)).$

Let $X \in KP^+ $ so that $Ad(t_k)X \in Sp_0(n,F).$ Condition (1.2) gives us the following equalities, $$ \bar A_1 -\bar B_1 +\bar D_1= A_1 +B_1 +D_1, \hskip 0.3cm B_2 +D_2 =\bar A_2, \hskip 0.3cm \bar A_3 -\bar B_3 =D_3, \hskip 0.3cm \bar A_4 = D_4$$ $$ \bar A_1 +\bar B_1 -\bar D_1= A_1 -B_1 -D_1, \hskip 0.3cm \bar B_2 -\bar D_2 =A_2, \hskip 0.3cm \bar A_3 +\bar B_3 = -D_3, \hskip 0.3cm B_4=0.$$ From the second equality on each line,  we deduce $D_2=0.$ Thus, $B_2 =\bar A_2.$    From the third equality in both lines we obtain $\bar A_3=0$. Hence $A_3=0$  and $B_3 =-\bar D_3$. Next $\tr A D -\tr B 0=I$ give us  $D= \tr A^{-1}.$ Explicitly  $D=(\tr A_1^{-1}, 0, -\tr (A_1^{-1}A_2A_4^{-1}), \tr A_4^{-1}).$ Since $(A,B,0,D) \in Sp(n,E)$ and so $\tr B D =\tr D B.$  The computation of the last equality lead us to $$ \begin{pmatrix} A_1^{-1} B_1 -\tr Y \bar Y & A_1^{-1} \bar A_2 \\ A_4^{-1}\bar Y & 0 \\ \end{pmatrix} = \begin{pmatrix} \tr B_1 \tr A_1^{-1}  -\tr \bar Y  Y &  \tr \bar Y \tr A_4^{-1}  \\ -\tr \bar A_2^{-1} \tr A_1^{-1} & 0 \\ \end{pmatrix} $$
where $Y:=\tr (A_1^{-1}A_2A_4^{-1})$

Now,  the equality of the (2,1)-coefficients  gives $A_4^{-1} \tr \bar A_4^{-1} \tr \bar A_2 \bar A_1^{-1} = -\tr \bar A_2 \tr  A_1^{-1},$  which, after we transpose both members of the last equality, we obtain $$ \bar A_1^{-1} \bar A_2 \bar A_4^{-1} \tr  A_4^{-1} = -A_1^{-1} \bar A_2. $$ From, equality of the (1,2)-coefficients implies $$  \bar A_1^{-1} \bar A_2 \bar A_4^{-1} \tr  A_4^{-1} = A_1^{-1} \bar A_2. $$ Thus, $A_2=0$ and we have that $$(A,B,0,D)= (diag(A_1, A_4), diag(B_1,0), 0, diag(\tr A_1^{-1}, \tr A_4^{-1})). $$ The hypothesis $Ad(t_k)(A,B,0,D) \in Sp_0(n,F)$ let us conclude that $A_1 \in O(k,F), \\ A_4 \in U(n-k,E)$. From here, (E) is shown, and the theorem follows.

 \section{Anti-involutions in $Sp(n,F).$}

In this section we analyze the structure on the set of anti-involutions in the group $Sp(n,F)$.  We will show that this set  is a homogeneous space for $Sp(n,F)$.\\
The denote by $\mathcal C(n,F)$ the set of anti-involutions ,i.e.,
 $$\mathcal C(n,F)=\{ T \in Sp(n,F): T^2 =-1 \}.$$
 \begin{prop}

 $\mathcal C(n,F)$ is equivariant isomorphic to $\mathcal H_n$ when $-1$ is not a square in $F,$ whereas is isomorphic to \\$Sp(n,F)/(Sp(n,F)\cap K) $ when $-1$ is a square in $F.$
 \end{prop}
    It is clear that  $\mathcal C(n,F)$ is invariant under conjugation. Since $J =(0, I_n, -I_n,0) $ is an element of $Sp(n,F)$ we have that $JT$ is an element of $Sp(n,F).$
   \smallskip
    The poof of the proposition will follow from the next three lemmas
    \smallskip

\begin{lem} i) Let $T$ be an involution, then $JT$ is a symmetric matrix. That is, $\tr(JT)=JT$

ii) For $T \in Sp(n,F),$  such that   $JT$ is  symmetric, we have that
$T$ is an involution.

\end{lem}

Proof:  Recall $\tr J=-J, \tr T J T= J, T^2 = -1$ Hence,  $ \tr (JT) = -\tr T J =-J T^{-1}=JT.$  For the second statement, we have $\tr(JT)=JT$ hence $ J= -\tr T^{-1} JT =\tr T JT $ thus $ T^2 =-I.$ \begin{flushright} $\Box$ \end{flushright}


\bigskip

According  to lemma 6, to each involution $T$ in $Sp(n,F)$ we naturally associate a symmetric non-degenerate bilinear form $b_T$ on $F^{2n}.$ The matrix of the form $b_T$  in the canonical basis is $JT.$

Now, from the classification of symmetric non-degenerate bilinear forms on $F^{2n}$ we have that $b_T$ is either equivalent to the Euclidean form  $x_1^2 +\dots +x_{2n}^2 $ or to the non-Euclidean form  $x_1^2 +\dots +x_{2n-1}^2+ c x_{2n}^2 $  where $c \in F$ is not a square.

Since $det(JT)=1.$
we obtain


 \begin{rem}The form  $b_T$ is always equivalent to the Euclidean form.

  The group $Sp(n,F)$ acts on $Sp(n,F) \cap Sym(F^{2n})$ by the formula  $$ (g,S) \rightarrow \tr (g^{-1})S g^{-1}.$$ It readily follows that the map $ \mathcal C(n,F) \ni T \rightarrow JT \in  Sp(n,F) \cap Sym(F^{2n})$ intertwines the respective actions of $Sp(n,F).$

  Hence, for $g \in Sp(n,F)$ the forms $b_T$  and $ b_{gTg^{-1}}$ are equivalent.
\end{rem}

\bigskip

To continue with, we split up the analysis of $\mathcal C(n,F)$  into the two possible cases, namely, $-1$ is either a square in $F$ or $-1$ is not a square in $F.$

\bigskip
\medskip

We assume first that  $-1$ is not an square in $F.$ Let us fix a square root $i \in E$ of $-1.$

\smallskip

For an anti involution $T\in Sp(n,F)$ we have that $T$ is a semisimple linear map with possible  eigenvalues $i, -i$ because the minimal polynomial of $T$ divides $x^2 +1.$

Let $V_i(T) $ (resp $V_{-i}(T)$) the corresponding possible eigenspace in $E^{2n}.$ Hence, $E^{2n}= V_i(T) \oplus V_{-i}(T),$ and we have

\begin{prop} i) Both subspaces $V_{i}(T), V_{-i}(T)$ are nonzero.

ii) $\overline{V_{i}(T)} = V_{-i}(T).$

iii) $F^{2n} \cap V_{i}(T) = F^{2n} \cap V_{-i}(T) =\{0\}.$

iv) The map  $ F^{2n} \ni v \rightarrow v -iTv \in V_{i}(T)$ is linear bijection over $F.$

v) $V_{i}(T) $ (resp $V_{-i}(T)$) is a lagrangian subspace.

vi)  $h_E (v -iTv, w -iTw) =2 \omega (v,w) + 2i b_T(v,w), \, \text{for} \, v , w \in F^{2n}. $

vii) The decomposition $E^{2n}= V_i(T) \oplus V_{-i}(T)$ is orthogonal with respect to $h_E.$

viii)  $h_E$ restricted to $V_i(T)$ is non degenerate.
\end{prop}

{\it Proof:} The result from the facts $T \in U(h_\mathbb R, E^{2n}) \cap Sp(n,E)$  and  $i \notin F. In particular, $  viii) follows from vii) and that $h_E$ is non degenerate.  For $x,y \in V_i(T), \omega(x,y)=\omega(Tx,Ty)=ii\omega(x,y)=-\omega(x,y).$  \begin{flushright} $\Box$ \end{flushright}

\bigskip
\smallskip


 Let $v_j -iTv_j, j=1, \dots, n$ denote an orthonormal basis of $V_i(T)$ for the restriction of $\frac{1}{2i} h_E .$
Then, $v_1, \dots, v_n$ span a lagrangian subspace of $F^{2n}$  and $v_1, \dots, v_n, Tv_1, \dots, Tv_n$ is a basis for $F^{2n}.$    \\    In fact, from vi) we obtain $ w(v_k, v_s)= 0,   b_T(v_k, v_s)= \delta_{k,s}$. The last statement follows from  $T^2 =-1$ applied to $\sum_{1 \leq j \leq n} c_j v_j + d_j Tv_j =0$ for $c_j, d_j \in F$ and a short computation.

\begin{lem} Assume $-1$ is not a square in $F$. Then, the action of $Sp(n,F)$ in  $\mathcal C(n,F)$ is transitive.

\end{lem}

\begin{proof}
Proposition 6 gives rise to a map from $\mathcal C(n,F) $ to $\mathcal L_{E,2n}$ by the rule  $$\mathcal C(n,F) \ni T \longrightarrow V_i(T)$$ From viii) we have the image of the map is contained in $\mathcal H_n.$ For $g \in Sp(n,F)$ we have the equality $ g V_i(T)=V_i(gTg^{-1}),$  which shows that the map is equivariant. The maps is obviously injective. Since $\mathcal H_n$ is an orbit of $Sp(n,F)$ (Theorem 1) we have that the map is a bijection and hence the result
\end{proof}

\smallskip
Next, we assume $-1=i^2 $ with $i \in F.$ Then, due to the semisimplicity of $T$ we have the decomposition $F^{2n}=(F^{2n} \cap V_i(T)) \oplus (F^{2n} \cap V_{-i}(T)).$

From the equalities $  \omega(x,y)=-\omega(x,y) \, \text{for}\, x,y \in V_i(T),$ we have that the subspaces $F^{2n} \cap V_i(T),  F^{2n} \cap V_{-i}(T)$ are isotropic, Corollary 3 pag 81 in \cite{Bourbaki}  gives us that both subspaces are lagrangian. Therefore,  the anti hermitian form $ h_E$ restricted to $F^{2n} \cap V_i(T)$ is the null form, which forces to $V_i(T)$ to be an element of $\mathcal H_0.$

\begin{lem} Assume $-1$ is a square in $F$. Then, $\mathcal C(n,F)$ is a homogeneous space equivalent to $Sp(n,F)/(Sp(n,F) \cap K).$
\end{lem}
\noindent
\begin{rem}The map $\mathcal C(n,F) \ni T \longrightarrow V_i(T) \in  \mathcal H_0$ is equivariant for $Sp(n,F)$ and in this case is no longer injective (c.f. example 3-a) , due to theorem 1 $\mathcal H_0$ is a homogeneous space for $Sp(n,F),$ hence, the map is surjective. 
\end{rem}
\noindent
We now show lemma 7. We set $$ H:=\begin{pmatrix}  iI_n  & 0\\
 0 &   - i I_n \\ \end{pmatrix}.$$ \noindent  Then, $H \in \mathcal C(n,F).$  Let $T$  be an anti involution in $Sp(n,F)$ we will show that $T$ is conjugated in $Sp(n,F)$ to the matrix $H.$ For this, we define
$D:=J^{-1}TJ,$  which is another anti involution in $Sp(n,F)$.

The minimal polynomial of $J^{-1}TJ$ divides the polynomial $x^2 +1=(x-i)(x+i)$ .    Hence, $D:=J^{-1}TJ$ is  diagonalizable over $F.$

Let $W_{\pm i}$ the associated eigenspaces.   Thus, $F^{2n}=W_i \oplus W_{-i}$.

 Since for every $v,w \in F^{2n}, \omega (Dv,Dw)=\omega (v,w),$ we have that $W_{\pm i}$ are isotropic subspaces for $\omega.$ The hypothesis that  $\omega$ is non degenerate forces,  $W_{\pm i}$ to be lagrangian subspaces.  Thus,  there exists $P \in Sp(n,F)$ so that

  $Pe_1, \dots, Pe_n$ is a basis for $W_i,$ \hskip 1.0cm $Pe_{n+1}, \dots, Pe_{2n}$ is a basis for $W_{-i}$


\noindent
  We have  \begin{center} $ D Pe_j = i Pe_j = P(ie_j)=P H(e_j), j=1, \dots, n, $

   $ D Pe_j =-i Pe_j=P(-ie_j)=PH(e_j)  , j=n+1, \dots, 2n.$

   \end{center}

   Hence, $DP = PH. $ That is, $$ PH = DP=J^{-1}T J P.$$ Therefore,

   $$ H=P^{-1} J^{-1} T J P=(JP)^{-1} T (JP).$$ The matrices in $Gl(2n,F)$ which commute  with $H$ are the matrices $diag(A,B),\\  A ,B, \in Gl_n(F).$  Thus, the isotropy at $H$ is $Sp(n,F)\cap K.$
\begin{flushright} $\Box$ \end{flushright}

\begin{rem}
 A particular element of $Sp(n,F)$ which conjugates $H$ onto $J$ is the Cayley transform $$ C(e_j)=\frac{1}{-2i} (e_j +i e_{n+j}), j=1, \dots, n, \hskip 0.5cm C(e_{n+j})=e_j -i e_{n+j}, j=1, \dots, n.$$

\end{rem}

\smallskip
\begin{example}
We assume $-1=i^2, i \in F.$

A simple calculation yields
$\mathcal C(1,F) $ is $$\{ \begin{pmatrix} \pm i & x\\
 0 &   -\pm i \\ \end{pmatrix},  \begin{pmatrix} \pm i & 0\\
 y  &   -\pm i \\ \end{pmatrix} , x\in F, y \in F^\times \} $$  union the set $$ \{ \begin{pmatrix} a & -\frac{1+a^2}{c}\\
 c &   -a \\ \end{pmatrix}, c \in F^\times, a \in  F\backslash \{\pm i \} \}$$

 \noindent Hence, the cardinal of the set of involutions is
 $2(q+q-1)+(q-2)(q-1)=  q(q+1).$  The isotropy at $diag(i,-i)$ is the subgroup $diag(a,-a), a \in F^\times.$ Hence $card(Sl(2,F_q))/card(F^\times)= q(q-1)(q+1)/(q-1)= card(\mathcal C(1,F)).$  Also,

$$V_i(\begin{pmatrix}  -i & 0\\
 x &    i \\ \end{pmatrix})=F \begin{pmatrix} 0 \\ 1 \\ \end{pmatrix} ,  \hskip 0.5cm  V_i(\begin{pmatrix}  i & x\\
 0 &   - i \\ \end{pmatrix})=F \begin{pmatrix} 1 \\ 0  \\\end{pmatrix}. $$

 $$V_i(\begin{pmatrix}  a & -\frac{1+a^2}{c}\\
 c &    -a \\ \end{pmatrix})=F \begin{pmatrix} \frac{1+a^2}{c} \\ a-i \\ \end{pmatrix},$$  $$V_i(\begin{pmatrix}  -i & x\\
 0 &    i \\ \end{pmatrix})=F\begin{pmatrix} x \\ 2i \\ \end{pmatrix} , V_i(\begin{pmatrix}  i & 0\\
 x &    -i \\ \end{pmatrix})=F\begin{pmatrix} 2i \\x \\ \end{pmatrix}. $$

\end{example}

\bigskip

\subsection{The case $T^2=a, a$ square}

Let $F$ be a field of odd characteristic, and let $\omega$ be a non degenerate alternating form in  $V=F^{2n}$ . We fix  $a \in F$ and define $$S_a :=\{T \in Sp(w): T^2=a Id\}$$

 for $a=1 $ the identity matrix belongs to $S_a$

 for $a=-1$ the matrix $J$ belongs to $S_a$

 \begin{prop} For $a \notin \{1,-1\}$ and $a=b^2, b \in F$ the set $S_a$ is empty.
 \end{prop}
 \begin{proof}
  Let $T \in S_a$ , then the eigenvalues of $T$ belongs to  the set $\pm b.$ Let $W_b, W_{-b}$ be the eigenspaces of $V.$

 The equality  $ \frac12 (bI-T) +\frac12 (bI +T)= b I$ implies that $V=W_b \oplus W_{-b}.$

 For $x, y \in W_b,$ we have $ \omega(x,y)=0$  ( $\omega(x,y)=\omega(Tx,Ty)=b^2 \omega(x,y),$ ) . Similarly,  for $x, y \in W_{-b}$ we have $ \omega(x,y)=0$. Therefore, both subspaces are isotropic.

 We now verify for $x, \in W_b, y \in W_{-b} $ that $ \omega(x,y)=0.$ In fact,  $ \omega(x,y)=\omega(Tx,Ty)=b (-b) \omega(x,y)=-a \omega(x,y).$ Since $a \not= -1$, we get $ \omega(x,y)=0.$

 Then, assuming $S_a$ is not empty,  unless $a \in \{1,-1\}$ we have $\omega$ equal to the null form, and the result follows.
 \smallskip

 Another proof follows along the following lines :

 For a symplectic matrix, if $\lambda $ is an eigenvalue, then $1/\lambda$ is also an eigenvalue.

 So if $b,-b$ are the unique eigenvalues, and $b \notin \{\pm 1, \pm i\}$  we must have $-b=1/b$ from which  $b^2=-1$ so $a=-1.$
\end{proof}

\subsubsection{The case $a=1$}

 Let $W$ be any subspace of $V$ such that $\omega$ restricted to $W$ is non degenerate, so $V=W \oplus W^\perp.$

 Define $T_W$ to be the linear operator equal to the identity in $W$ and equal to $-I$ in $W^\perp.$

It readily follows that $T_W \in Sp(n,F)$ and  $T_W$ is an involution.

\begin{prop} Any involution $T$ in $Sp(n,F)$ is equal to a $T_W$ for a convenient $W.$
\end{prop}
\begin{proof}
 In fact,   the eigenvalues of $T$ belongs to  the set $\pm 1$ Let $W_1, W_{-1}$ be the eigenspaces of $V$
 the equality  $ \frac12 (I-T) +\frac12 (I +T)=  I$ implies that $V=W_1 \oplus W_{-1}.$

 For $x, \in W_1, y \in W_{-1} $ we have $ \omega(x,y)=0.$ In fact, $ \omega(x,y)=\omega(Tx,Ty)= 1(-1) \omega(x,y)=-1 \omega(x,y).$

 It follows: $\omega$ restricted to any of the subspaces in non degenerate. Hence, $T=T_{W_1}.$
 \end{proof}

\begin{cor} The orbits of $Sp(n,F)$ in $\mathcal C_{1}(n,F)$ are parameterized by  $k=1,2, \dots, 2n.$ Indeed, for each $k$ the set of involutions $T$ such that its $1-$eigenspace is of dimension $k, $ is an orbit for $Sp(n;F).$
\end{cor}

 \section{Acknowledgements}
 Part of the work was done during the meeting "Representation theory days in Patagonia" organized by University of Talca. The authors are grateful to the organizers of the workshop, Stephen Griffeth, Steen Ryom-Hansen, Jean F. van Diejen, Luc Lapointe, for their kind invitation to attend to such a nice meeting.

The authors want also to thank  Pierre Cartier for  illuminating discussions and suggestions related to this work.


\begin{thebibliography}{999}

\bibitem{Bourbaki} Bourbaki, N., \'El\'ements de Math\'ematique, Livre II, Alg$\grave{e}$bre, Chapter IX, Herman et Cie, 1959.

\bibitem{dieu} Dieudonne, J.,  Le G\'eom\'etrie des groupes classiques, Ergebnisse der Mathematik und ihrer Grenzgebiete, Springer Verlag , 1955.

\bibitem{rrs} Roger Richardson, Gerhard R\"{o}hrle, and Robert Steinberg Parabolic,. Parbolic  subgroups with Abelian unipotent radical. Invent Math. 110, 649-67l (1992)

\bibitem{siegel} Siegel, C., Symplectic geometry, American Journal of Mathematics, Vol. 65, No. 1  , pp. 1-86,(Jan., 1943).

\bibitem{kaneyuki} Kaneyuki, S., Pseudo-hermitian symmetric spaces and Siegel domains over non degenerate cones, Hokkaido Math. Jour. Vol 20, 213-239, (1991).
\bibitem{jsatesis} Soto Andrade, J.,     Bull. Soc. Math. France

\bibitem{wolf2} Wolf, J., Fine Structure of Hermitian Symmetric spaces,  Symmetric Spaces, short courses presented at Washington University, (Boothby, Weiss, edts), pages 271-357, Marcel Dekker, Inc, New York, 1972.

\bibitem{psa} Pantoja, J. and Soto-Andrade, J., Repr\'{e}sentations de $SL_{*}(2,A)$ et $SL(n,q)$, C.R. Acad. Sci. Paris, t. 323 S\'{e}rie I, p. 1109-1112, 1996.

\bibitem{jaja} Soto-Andrade, J. and Vargas, J., Twisted spherical functions on the Poicar\'{e} Upper Half Plane, J. Algebra 248, 724-246, 2002.

\bibitem{psa2} Pantoja, J. and Soto-Andrade, J., A Bruhat decomposition of the group $SL_{*}(2,A)$, J. Algebra 262, 401-412, 2003.

\bibitem{takeuchi} Takeuchi, M., On Orbits in a Compact Hermitian Symmetric Space, American Journal of Mathematics, Vol. 90, No. 3 (Jul. 1968), 657-680.




\end{thebibliography}
\end{document}